\documentclass[12pt, reqno]{amsart}
\usepackage{ifthen,amsfonts,
amsmath,amssymb,
epsfig,color, url}

\begingroup

\newtheorem{theorem}{Theorem}[section]

\newtheorem{propos}[theorem]{Proposition}
\newtheorem{corol}[theorem]{Corollary}
\endgroup

\newtheorem{remark}[theorem]{Remark}

\numberwithin{equation}{section}

\newcommand{\eps}{{\varepsilon}}

\newcommand{\vol}{{\text{vol}}}
\newcommand{\V}{{\text{V}}}
\newcommand{\Ric}{{\text{Ric}}}
\newcommand{\Rico}{{\stackrel{\circ}{\text{Ric}}}}

\newcommand{\R}{{\text{R}}}
\newcommand{\Hess}{{\text {Hess}}}

\def\SS{{\mathbb  S}}

\begin{document}

\title{Almost-Schur lemma}

\author{Camillo De Lellis}
\address{Institut f\"ur Mathematik, Universit\"at Z\"urich,
8057 Z\"urich, CH}
\email{camillo.delellis@math.unizh.ch}
\author{Peter M. Topping}
\address{Mathematics Institute,
University of Warwick,
Coventry, CV4 7AL, UK}
\email{P.M.Topping@warwick.ac.uk}

\begin{abstract}
Schur's lemma states that every Einstein manifold of dimension
$n\geq 3$ has constant scalar curvature.
In this short note we ask to what extent the scalar curvature
is constant if the traceless Ricci tensor is assumed to
be \emph{small} rather than identically zero. In particular,
we provide an optimal $L^2$ estimate under suitable assumptions
and show that these assumptions cannot be removed.
\end{abstract}

\maketitle

\section{Introduction}

Schur's lemma states that every Einstein manifold of dimension
$n\geq 3$ has constant scalar curvature.
Here $(M,g)$ is defined to be Einstein if its traceless
Ricci tensor 
$$\Rico:=\Ric-\frac{R}{n}g$$ 
is identically zero.

In this short note we ask to what extent the scalar curvature
is constant if the traceless Ricci tensor is assumed to
be \emph{small} rather than identically zero. As it is customary,
we say that $M$ is a closed manifold if it is compact and without
boundary.

\begin{theorem}\label{t:main} 
For any integer $n\geq 3$, if $(M,g)$ is a closed 
Riemannian manifold of dimension $n$ with nonnegative
Ricci curvature, then 
\begin{equation}\label{e:mainest}
\int_M \left(R-\overline R\right)^2 
\;\leq\; \frac{4n(n-1)}{(n-2)^2}
\int_M 
|\Rico|
^2\, 
\end{equation}
where $\overline R$ is the average value of $R$ over $M$.
Moreover the equality holds if and only if $M$ is Einstein.
\end{theorem}

Since
\begin{equation}\label{e:split}
\left|\Ric - \textstyle{\frac{\overline R}{n}}\,
g\right|^2
\;=\; |\Rico|^2
+\frac{1}{n}\left(R-\overline R\right)^2,
\end{equation}
we immediately get:

\begin{corol}
Under the same conditions as in the theorem,
\begin{equation} \label{e:secondest}
\int_M \left|\Ric - \textstyle{\frac{\overline R}{n}}\,
g\right|^2
\;\leq\; \frac{n^2}{(n-2)^2} \int_M \left|\Ric - 
\textstyle{\frac{R}{n}}\, g\right|^2\, ,
\end{equation}
where the equality holds if and only if $M$ is Einstein.
\end{corol}
These estimates are sharp in the following senses.

First, the constants are the best possible because
if we were to reduce either constant the inequalities
would fail for certain small but high-frequency deformations
of the round sphere as we discuss in Section \ref{secvarsect}.
Indeed, if $g_0$ is the metric of the round sphere
then we can take a conformal deformation $(1+f)g_0$
where $f$ is an eigenfunction of the Laplacian on the
sphere corresponding to a suitably large eigenvalue.

Second, the curvature condition $\Ric\geq 0$ cannot
simply be dropped, as we discuss in Section \ref{negRicsect}:
For $n\geq 5$, we show that any such inequality then fails 
even if we restrict $M$ to be diffeomorphic to the sphere. 
For example, we can find metrics $g$ on $S^n$ which
make the ratio of the left-hand side of \eqref{e:secondest}
to the right-hand side of \eqref{e:secondest} arbitrarily 
large.
If we are able to prescribe the topology of $M$, then
the same thing can be engineered even in dimension $n=3$:
we can find manifolds $(M^3,g)$ which make the same ratio
arbitrarily large.
We leave open the possibility that inequalities of
this form may hold for $n=3$ and $n=4$ with constants
depending on the topology of $M$. 
We finally mention that Ge and Wang in \cite{GW} have followed up on the first
version of this paper by demonstrating that for four dimensional
manifolds our hypothesis can be weakend to nonnegative
\emph{scalar} curvature. This is surely not possible for $n\geq 5$
(as can be shown using constructions similar to the ones of 
Section \ref{negRicsect}), whereas the case $n=3$ is still open.

In the context of the sectional Schur's lemma, two results
which are somewhat related to ours have appeared in \cite{Nikolaev}
and \cite{Gribkov}.
However, of all known inequalities which generalise a
geometric rigidity statement, the closest one to our result seems
to be the inequality of
M\"uller and the first author \cite{delellis_mueller},
which generalises the
well-known assertion that the only totally umbilic
closed surfaces of the Euclidean three dimensional
space are spheres.
In fact our method also gives that result 
with the sharp constant for convex hypersurfaces of any
dimension, even within more general Einstein ambient 
manifolds; details will appear in \cite{danielPhDthesis}. 
As with the proof in this paper, our method in that case
has the advantage of being completely elementary, whereas
the proof of 
\cite{delellis_mueller}
exploits deep tools from partial differential equations and its
only advantage is that it holds for general smooth surfaces.

\section{Proof of Theorem \ref{t:main}}

\subsection{Proof of \eqref{e:mainest}}
Recall that the contracted second Bianchi 
identity tells us that $\delta \Ric + \frac12 dR=0$
(where $(\delta\Ric)_j:=-\nabla_i R_{ij}$) and
hence that
\begin{equation}
\label{e:bianchi}
\delta \Rico = -\frac{n-2}{2n} dR.
\end{equation}
Let $f:M\to\R$ be the unique solution to
\begin{equation}
\left\{
\begin{aligned}
\Delta f & = R-\overline R \\
\int_M f & = 0.
\end{aligned}
\right.
\end{equation}
We may then compute
\begin{equation}\label{e:est1}
\begin{aligned}
\int_M \left(R-\overline R\right)^2 
&= \int_M \left(R-\overline R\right)\Delta f
= -\int_M \langle dR, df\rangle \\
&= \frac{2n}{n-2}
\int_M \langle \delta\Rico, df\rangle
= \frac{2n}{n-2}
\int_M \langle \Rico, \Hess f\rangle\\
&= \frac{2n}{n-2}
\int_M \langle \Rico, \Hess f - {\textstyle \frac{\Delta f}{n}}g\rangle\\
&\leq \frac{2n}{n-2}\| \Rico \|_{L^2}
\| \Hess f - {\textstyle\frac{\Delta f}{n}}g\|_{L^2}.
\end{aligned}
\end{equation}
Now by integration by parts (i.e. the Bochner formula) 
we know that
\begin{equation}
\label{bochner}
\int_M | \Hess f |^2 = \int_M (\Delta f)^2 
- \int_M \Ric(\nabla f,\nabla f)
\end{equation}
and therefore
\begin{equation}\label{e:Bochner}
\begin{aligned}
\int_M | \Hess f - {\textstyle \frac{\Delta f}{n}}g|^2 
&= \int_M | \Hess f |^2 - \frac{1}{n}(\Delta f)^2\\
&= \frac{n-1}{n}\int_M (\Delta f)^2 
- \int_M \Ric(\nabla f,\nabla f)\\
&= \frac{n-1}{n}\int_M (R-\overline R)^2 
- \int_M \Ric(\nabla f,\nabla f),
\end{aligned}
\end{equation}

and since the Ricci curvature is nonnegative,
we have
\begin{equation}\label{e:est2}
\| \Hess f - {\textstyle \frac{\Delta f}{n}}g\|_{L^2} \leq \left(\frac{n-1}{n}\int_M \left(R-\overline R\right)^2\right)^\frac12,
\end{equation}
which can be combined with \eqref{e:est1} to give \eqref{e:mainest}.

\begin{remark}
We note that the Ricci term which
we throw away in the proof does not destroy optimality
because that term is `lower order' -- i.e. only involves
first derivatives -- and is thus insignificant for very
high frequency $f$.
\end{remark}

\begin{remark}
We only use the Ricci hypothesis in the proof in order to obtain
the $L^2$ estimate
\begin{equation}
\int_M | \Hess f |^2 \leq \int_M (\Delta f)^2.
\end{equation}
Moreover, a slight adaptation of the proof would establish
an $L^p$ version of our results on any manifold 
supporting a Calderon-Zygmund inequality
\begin{equation}
\int_M | \Hess f |^p \leq C\int_M (\Delta f)^p.
\end{equation}
\end{remark}

\subsection{Equality} Obviously, if $M$ is an Einstein manifold,
then both sides of \eqref{e:mainest} vanish. Assume next that 
$M$ satisfies the equality in \eqref{e:mainest}. Then equality
must hold in \eqref{e:est1} and \eqref{e:est2}. Equality holds
in the latter inequality if $\Ric (\nabla f, \nabla f)=0$
(see \eqref{e:Bochner})
and since $\Ric\geq 0$, 
\begin{equation}
\label{e:kernel}
\Ric(\nabla f,\cdot)=0.
\end{equation}
Meanwhile, equality holds in \eqref{e:est1} if and only if
the two tensors $\Rico$ and 
$\Hess f - \textstyle{\frac{\Delta f}{n}}\,g$ are linearly
dependent. If one of them vanishes, than \eqref{e:est1}
implies that $R$ is constant and hence, since equality in \eqref{e:mainest}
holds, that $M$ is Einstein. Otherwise, there is $\mu>0$ such that $\mu\Rico =  (\Hess f - \textstyle{\frac{\Delta f}{n}}\,g)$.
This, together with \eqref{e:kernel} and \eqref{e:bianchi}
implies that
\begin{equation}\label{e:delta}
- \frac{n-1}{n} d \Delta f \;=\; \delta \left(\Hess f - \frac{\Delta f}{n} g \right) \;=\; 
\mu\,  \delta \Rico \;=\; - \mu \frac{n-2}{2n} dR\, .
\end{equation}
Since $\Delta f = R-\bar{R}$, from \eqref{e:delta} we conclude
\begin{equation}\label{e:vanishes?}
\left(\frac{n-2}{2n} \mu - \frac{n-1}{n}\right) dR \;=\; 0\, .
\end{equation}
Thus $R$ is a constant (and hence $M$ is Einstein) unless
$\mu = \textstyle{\frac{2n-2}{n-2}}$.
Assuming this is the case, then
\begin{equation}\label{e:special2}
\Hess f - \frac{\Delta f}{n} g \;=\; \frac{2n-2}{n-2} \Rico\, .
\end{equation}
Combining \eqref{e:kernel} with \eqref{e:special2} and the identity $\Delta f = R-\bar{R}$
we infer
$$
\Hess f (\nabla f,\cdot) - \frac{R-\bar{R}}{n} df = 
- \frac{2n-2}{(n-2)n} R \,df. 
$$
Rewrite this last identity as
\begin{equation}\label{e:special3}
\nabla \frac{|\nabla f|^2}{2} = - \left[\frac{\bar{R}}{n} + \frac{R}{n-2}\right] \nabla f\, .
\end{equation}
Fix a point $x_0\in M$ and let
$\gamma:[0, +\infty[ \to M$ be the solution of 
$\dot{\gamma} (t) = - \nabla f (\gamma (t))$ with $\gamma (0) = x_0$.
Consider $\alpha(t) = f (\gamma (t))$. Then $\alpha'(t)= - |\nabla f (\gamma (t))|^2$ and, by \eqref{e:special3}, 
$$
\alpha''(t) = - 2 \left[\frac{\bar{R}}{n} + \frac{R}{n-2}\right]
|\nabla f (\gamma (t))|^2\;\leq\; 0\, .
$$
Thus, $\alpha$ is a bounded nonincreasing concave 
function on $[0, +\infty[$
and therefore it must be constant. We conclude that
$- |\nabla f (x_0)|^2 = \alpha'(0)=0$. The arbitrariness of $x_0$
implies that $f$ is constant which completes the proof.

\section{Second variation arguments}\label{secvarsect}

We will show that the constants in \eqref{e:mainest} and
\eqref{e:secondest} are optimal. We do this by computing
the second variation formula of each side of the
inequalities based at the round sphere of dimension
$n\geq 3$. If the constant in
either inequality were reduced at all, then we could find 
small, high-frequency perturbations of the round sphere
which violated both estimates.

\begin{proof}[Optimality of \eqref{e:mainest}
and \eqref{e:secondest}] First of all observe that, by
\eqref{e:split}, the optimality of one inequality
implies the optimality of the other. We next
consider the standard sphere $M= (\SS^n, \sigma)$ 
for which $\Ric=(n-1)\sigma$ and $R=n(n-1)$, and
deform it through a one-parameter family of Riemannian manifolds
$M_t = (\SS^n, g_t)$ where $g_t = (1+tf) \sigma$. We assume
that $f\in C^\infty (M)$ and $\int_M f = 0$. Set
\begin{eqnarray}
F(t) &:=& C \int_{M_t} \left|\Ric - 
\textstyle{\frac{R}{n}}\, g\right|^2 - \int_{M_t} \left|\Ric - 
\textstyle{\frac{\bar{R}}{n}}\, g\right|^2\nonumber\\
&=& (C-1) \int |\Ric|^2 - \frac{C}{n} \int \R^2 + \frac{1}{n\V}
\left( \int \R\right)^2\nonumber\\
&=:& (C-1) F_1 (t) - \frac{C}{n} F_2 (t) + \frac{1}{n} F_3 (t)
\label{e:123}
\end{eqnarray}
where $V$ is the volume of $M_t$. We write $d \vol$ for the 
volume element.
Straightforward calculations (see for instance Section 2.3.1 of
\cite{RFnotes}) give
\begin{eqnarray}
\left.\partial_t d \vol\right|_0 &=& \frac{n}{2} f\, d \vol\label{e:volf}\\
\left.\frac{d}{dt} \V\right|_0 &=& 0\label{e:Vvanf}\\
\left. \partial_t g^{ij}\right|_0 &=& -f \sigma^{ij} \label{e:gupf}\\
\left.\partial_t \Ric_{ij}\right|_0 &=& - \frac{1}{2} \left(
\Delta f \sigma_{ij} + (n -2) f_{;ij} \right) 
 \label{e:Ricf}\\
\left. \partial_t \R\right|_0 &=& 
- (n-1) \Delta f - (n-1) n f\label{e:Scalf}
\end{eqnarray}
\begin{eqnarray}
\left.\frac{d}{dt} \int \R\right|_0
&=& 0\label{e:iScalf}\, .
\end{eqnarray}
Therefore $F'(0)=0$. We next show
that, for any constant $C< n^2 (n-2)^{-2}$, there is a choice of $f$
such that $F''(0)<0$. This will imply
the optimality of \eqref{e:secondest} as desired.

We start by remarking that
\begin{eqnarray}
F''_2 (0) &=& \frac{d}{dt} \left( 2 \int \R\, \partial_t\R
+ \int \R^2\, \partial_t d \vol\right)\nonumber\\
&=& 2 n (n-1) \int \partial^2_t \R 
+ 2 \int (\partial_t\R)^2\nonumber\\
&&+ 4 n (n-1) \int \partial_t \R\,
\partial_t d \vol + n (n-1) \int \R\, \partial_t^2
d \vol\nonumber\\
&=& 2 n (n-1) \frac{d^2}{dt^2} \int \R +2 \int (\partial_t\R)^2
- n^2 (n-1)^2 \frac{d^2 V}{dt^2}\, .\label{e:F2''}
\end{eqnarray}
Similarly,
\begin{eqnarray*}
F_3'' (0) &=& \frac{d}{dt} \left( - \frac{1}{\V^2} \frac{d\V}{dt}
\left(\int \R\right)^2 
+ \frac{2}{V} \int \R \frac{d}{dt} \int \R\right)\nonumber\\
&=& - n^2 (n-1)^2 \frac{d^2\V}{dt^2}
- \underbrace{\frac{dV}{dt}\frac{d}{dt} 
\left(\frac{1}{\V^2} \left(\int \R\right)^2\right)}
_{\mbox{$=0$ by \eqref{e:Vvanf}}} 
\end{eqnarray*}
\begin{eqnarray}
&&+ \underbrace{\frac{d}{dt} \left(\frac{2}{V} \int \R \right)
\frac{d}{dt} \int \R}_{\mbox{$=0$ by \eqref{e:iScalf}}}
+ 2 n (n-1) \frac{d^2}{dt^2} \int \R\nonumber\\
&=& - n^2 (n-1)^2 \frac{d^2\V}{dt^2} +
2 n (n-1) \frac{d^2}{dt^2} \int \R\, . \label{e:F3''}
\end{eqnarray}
Finally we compute
\begin{equation}\label{e:F1''(a)}
F_1'' (0) \;=\; \int \partial^2_t |\Ric|^2
+ 2 \int \partial_t |\Ric|^2\, \partial_t d \vol + 
\int |\Ric|^2\, \partial^2_t d \vol\, .
\end{equation} 
Note that 
\begin{eqnarray}
\left.\partial_t |\Ric|^2\right|_0 
&=& 2 \partial_t \Ric_{ij} 
\Ric_{kl} g^{ik} g^{jl} + 2 \Ric_{ij} \Ric_{kl}
\partial_t g^{ik} g^{jl}\nonumber\\
&=& 2 (n-1) \left. \partial_t R\right|_0\, .\label{e:F1''(b)}
\end{eqnarray}
\begin{eqnarray}
\left. \partial^2_t |\Ric|^2\right|_0
&=& 2 \partial_t \left[ \partial_t \left(\Ric^{ij} g_{jl}\right)
\Ric_{i\alpha} g^{\alpha l}\right]\nonumber\\
&=& 2 (n-1) \left. \partial^2_t \R\right|_0
+ 2 \left[\partial_t \left(\Ric_{ij} g^{jl}\right)
\partial_t\left (\Ric_{l\alpha} g^{\alpha i}\right)\right]\nonumber\\
&=& 2 (n-1) \left. \partial^2_t \R\right|_0
- 4 (n-1) f\, \partial_t\Ric_{ij} \sigma^{ij}\nonumber\\
&& + 2 |\partial_t\Ric|^2 + 2 n (n-1)^2 f^2\label{e:F1''(c)} 
\end{eqnarray}
Therefore, we conclude
\begin{eqnarray}
F_1'' (0) &=& 2 (n-1) \int \partial^2_t \R + 2 \int |\partial_t
\Ric|^2 - 4 (n-1) \int f\, \partial_t \Ric_{ij} \sigma^{ij}\nonumber\\
&&+ 2 n(n-1)^2 \int f^2 + 4 (n-1) \int \partial_t \R\,
\partial_t d \vol \nonumber\\
&& + (n-1) \int \R\, \partial^2_t d \vol\nonumber\\
&=& 2 (n-1) \frac{d^2}{dt^2} \int \R + 2 \int |\partial_t\Ric|^2
- 4 (n-1) \int f\, \partial_t \Ric_{ij} \sigma^{ij}\nonumber\\
&&+ 2 n (n-1)^2 \int f^2 - n (n-1)^2 \frac{d^2\V}{dt^2}\label{e:F1''}
\end{eqnarray}
Putting together \eqref{e:F1''}, \eqref{e:F2''} and
\eqref{e:F3''} we get
\begin{eqnarray}
&&F''(0)\;=\; -\frac{2C}{n}\int (\partial_t R)^2
+ 2 (C-1) \int |\partial_t \Ric|^2\nonumber\\ 
&&- 4 (C-1) (n-1)
\int f\, \partial_t \Ric_{ij} \sigma^{ij} + 2 (C-1) n (n-1)^2
\int f^2\, . \qquad\mbox{}\label{e:F''}
\end{eqnarray}
Next, we have
\begin{eqnarray}
\int (\partial_t R)^2 &=&
(n-1)^2 \left(\int (\Delta f)^2 - 2n \int |df|^2 
+ n^2 \int f^2\right)\label{e:Scaltf}\\
\int f\, \partial_t \Ric_{ij} \sigma^{ij} 
&=& (n-1) \int |d f|^2\label{e:Ric*hf}\\
\int |\partial_t \Ric|^2 &=&
\frac{n}{4} \int (\Delta f)^2 + \frac{n-2}{2}
\int (\Delta f)^2 + \frac{(n-2)^2}{4} \int |D^2 f|^2
\nonumber\\
&=& \frac{n (n-1)}{4} \int (\Delta f)^2 - \frac{(n-2)^2 (n-1)}{4}
\int |df|^2\, .\label{e:Rict^2f}
\end{eqnarray}
(where in the last line we used the Bochner formula \eqref{bochner}).
Assume now that $C= n^2 (n-2)^{-2} -\eps$ for some positive $\eps$. 
Inserting \eqref{e:Scaltf}, \eqref{e:Ric*hf}
and \eqref{e:Rict^2f} into \eqref{e:F''},
we conclude
\begin{equation}\label{e:F''bis}
F''(0)
\;\leq\; -a(n)\eps \int (\Delta f)^2 
+ b(n, \eps) \int |df|^2 + c(n,\eps) \int f^2\, ,
\end{equation}
where the constant $a$ is strictly positive (since $n\geq 3$). 
By choosing $f$ to be an eigenfunction of the Laplacian with 
sufficiently large eigenvalue, we then have $F'' (0)<0$
as desired.
\end{proof}

\section{Counterexamples without the hypothesis $\Ric\geq 0$.}
\label{negRicsect}

Our results assume we are working on a manifold of nonnegative
Ricci curvature.
We now wish to ask when we have a hope of proving an inequality
of the form
\begin{equation}
\label{Cest}
\int_M \left(R-\overline R\right)^2 
\;\leq\; C\int_M |\Rico|^2\, 
\end{equation}
on more general manifolds $(M,g)$.
\begin{propos}
\label{n5prop}
For any $C<\infty$ and integer $n\geq 5$, there exists a
metric $g$ on the sphere $S^n$ such that \eqref{Cest}
fails.
\end{propos}
For smaller $n$, we know counterexamples only when
the topology of $M$ is allowed to depend on $C$:
\begin{propos}
\label{n3prop}
For any $C<\infty$, there exists a closed $3$-manifold
$(M,g)$ such that \eqref{Cest} fails.
\end{propos}

\begin{proof} (Proposition \ref{n5prop}.)
All we have to do is to
connect two round spheres of radii $1$ and $2$, say,
by a small neck.
On the two spherical parts, the traceless Ricci tensor 
$\Rico$ is zero. Therefore (for any $C$) we can make the
right-hand side of \eqref{Cest} as small as desired
for $n\geq 5$, since by scaling down the size of the neck, 
the integral of $|\Rico|^2$ over the neck will also be scaled down
to as small a value as we wish.
Meanwhile,
the different radii of the spherical parts ensure that 
the scalar curvature $R$ is
different on each sphere, and thus the left-hand side of
\eqref{Cest} cannot be small.
\end{proof}

\begin{proof} (Proposition \ref{n3prop}.)
This construction is loosely related to the one above.
The basic building block is any hyperbolic (constant sectional
curvature $-1$) $3$-manifold $(N,h)$ which fibres over the circle.
A result of Thurston implies that if $S$ is a closed 
surface of genus at least $2$, then the $3$-manifold arising
by gluing the boundary components of $[0,1]\times S$ using
a pseudo-Anosov diffeomorphism of the fibre $S$ must admit
a hyperbolic metric (\cite{thurston}).

Let us write $N^m$ for the $m$-fold covering of $N$ obtained
by taking covers of the base circle, and lift
the metric $h$ to a metric $\tilde h$ on $N^m$.
We also pick a point $p$ in $N$ and any one point $\tilde p$ in each
$N^m$ which projects to $p$ under the covering.
The idea then, for each $m\in \mathbb{N}$, is to attach one
$(N^m,\tilde h)$ to another scaled copy $(N^m,2\tilde h)$ via an $m$-independent
neck attached to small neighbourhoods of $\tilde p$ in each
$N^m$, to give a new manifold $(M,g)$.
With this construction, the right-hand side of \eqref{Cest}
is independent of $m$, but the left-hand side will increase
without bound as $m\to\infty$ at an asymptotically linear rate.
\end{proof}

\emph{Acknowledgements.} We thank Vlad Markovic for a useful 
converstation. CDL was supported by a grant of the Swiss National
Foundation.
PMT was supported by The Leverhulme Trust.

\end{document}